\documentclass[10pt]{article}
\usepackage{amssymb}
\usepackage{amsmath}
\usepackage{amsfonts}
\usepackage{amsthm}
\usepackage{graphics,graphicx}
\usepackage{float}
\usepackage{mathtools}
\usepackage[dvipsnames]{xcolor}

\usepackage[english]{babel}
\usepackage{authblk}

\newtheorem{theorem}{Theorem}

\newtheorem{remark}{Remark}

\newtheorem{o-problem}{Open problem}

\title{On reduction for eigenfunctions of graphs\thanks{This work was funded by the Russian Science Foundation under grant 22-21-20018}}
\author[1]{Alexandr Valyuzhenich\thanks{Chelyabinsk State University, Chelyabinsk, Russia; Email address: graphkiper@mail.ru}}

\date{}
\begin{document}
\maketitle

\begin{abstract}
In this work, we prove a general version of the reduction lemmas for eigenfunctions of graphs admitting involutive automorphisms of a special type.
\end{abstract}

\section{Introduction}\label{Sec:Intro}
Recently, for the eigenspaces of the Hamming and Johnson graphs, reduction lemmas were established (see \cite[Lemma 1]{V17} and \cite[Lemma 1]{VMV18}).
In \cite{EGGV22,MV20,V17,VV19,V21,VMV18,V20,V20Rec}, these lemmas were applied to study eigenfunctions and equitable $2$-partitions of the Hamming and Johnson graphs.
In this work, we generalize the reduction lemmas to graphs admitting involutive automorphisms of a special type.
In particular, we prove that an analogue of the reduction lemmas holds for the halved $n$-cube.

The paper is organized as follows. In Section \ref{Sec:Definitions}, we introduce basic definitions.
In Section \ref{Sec:Reduction}, we prove a general version of the reduction lemmas.
Then, in Section \ref{Sec:Examples}, we apply this result to the Hamming graph, the Johnson graph and the halved $n$-cube.
\section{Basic definitions}\label{Sec:Definitions}
Let $G$ be a graph. The vertex set of $G$ is denoted by $V(G)$.
Given a vertex $x$ of $G$, denote by $N_G(x)$ the set of all neighbors of $x$ in $G$.
For a set $W\subseteq V(G)$, denote by $G[W]$ the subgraph of $G$ induced by $W$.
The automorphism group of $G$ is denoted by $\mathrm{Aut}(G)$.

The eigenvalues of a graph are the eigenvalues of its adjacency matrix.
Let $G$ be a graph and let $\lambda$ be an eigenvalue of $G$.
A function $f:V(G)\longrightarrow{\mathbb{R}}$ is called a {\em $\lambda$-eigenfunction} of $G$ if $f\not\equiv 0$ and the equality
\begin{equation}\label{Eq:Eigenfunction}
\lambda\cdot f(x)=\sum_{y\in{N_G(x)}}f(y)
\end{equation}
holds for any vertex $x\in V(G)$. The set of functions $f:V(G)\longrightarrow{\mathbb{R}}$ satisfying (\ref{Eq:Eigenfunction}) for any vertex $x\in V(G)$ is called a {\em $\lambda$-eigenspace} of $G$. Denote by $U_{\lambda}(G)$ the $\lambda$-eigenspace of $G$.

Let $G$ be a graph. Let $\varphi$ be an automorphism of $G$ and let $\{V_1, V_2, V_3\}$ be a partition of $V(G)$.
The pair $(\varphi,\{V_1, V_2, V_3\})$ is called {\em special} if the following conditions hold:
\begin{enumerate}
    \item $\varphi(V_1)=V_2$ and $\varphi(V_2)=V_1$, i.e. $\varphi$ swaps $V_1$ and $V_2$.
    
    
    \item For any vertex $x\in V_i$, where $i\in \{1,2\}$, it holds $N_G(x)\cap V_{3-i}=\{\varphi(x)\}$.
    
    \item $\varphi(x)=x$ for any vertex $x\in V_3$, i.e. $\varphi$ stabilises $V_3$ pointwise.
\end{enumerate}

\begin{remark}\label{RemarkBitrade}
If  $(\varphi,\{V_1, V_2, V_3\})$ is a special pair of a graph $G$, then the following properties hold:
\begin{itemize}
  \item The graphs $G[V_1]$ and $G[V_2]$ are isomorphic.
  \item The graph $G[V_1\cup V_2]$ is isomorphic to the Cartesian product of $G[V_1]$ and $K_2$.
  \item The automorphism $\varphi$ is involutive.
\end{itemize}
\end{remark}

Suppose $P=(\varphi,\{V_1, V_2, V_3\})$ is a special pair of a graph $G$.
Let $G[V_1]$ and $G[V_2]$ be isomorphic to a graph $G_0$ and let $\varphi_1:V_1\longrightarrow V(G_0)$ and $\varphi_2:V_2\longrightarrow V(G_0)$ be the corresponding isomorphisms. Given a function $f:V(G)\longrightarrow{\mathbb{R}}$, we define a function $f_{P,\varphi_1,\varphi_2}$ on the vertices of $G_0$ as follows:
$$f_{P,\varphi_1,\varphi_2}(x)=f(\varphi_{1}^{-1}(x))-f(\varphi_{2}^{-1}(x)).$$

Let $\{i_1,\ldots,i_k\}$ be a subset of $\{1,2,\ldots,n\}$, where $1\leq k<n$.
For a vector $x\in \mathbb{Z}_{q}^n$, denote by $\Delta_{i_1,\ldots,i_k}(x)$ the vector obtained from $x$ by deleting coordinates with numbers $i_1,\ldots,i_k$.

Let $i,j\in \{1,2,\ldots,n\}$ and $i<j$. For a vector $x\in \mathbb{Z}_{q}^n$, denote by $\pi_{i,j}(x)$ the vector obtained from $x$ by interchanging the ith and jth coordinates.

\section{Reduction for eigenfunctions of graphs}\label{Sec:Reduction}
In this section, we prove the main theorem of this paper.
\begin{theorem}\label{Th:Reduction}
Suppose $P=(\varphi,\{V_1, V_2, V_3\})$ is a special pair of a graph $G$.
Let $G[V_1]$ and $G[V_2]$ be isomorphic to a graph $G_0$ and let $\varphi_1:V_1\longrightarrow V(G_0)$ and $\varphi_2:V_2\longrightarrow V(G_0)$ be the corresponding isomorphisms. If $f$ is a $\lambda$-eigenfunction of $G$, then $f_{P,\varphi_1,\varphi_2}\in U_{\lambda+1}(G_0)$.
\end{theorem}
\begin{proof}
For every $i\in \{1,2,3\}$, denote $G_i=G[V_i]$. Define a function $h$ on the vertices of $G$  as follows:
$$h(x)=f(x)-f(\varphi(x)).$$
Since $f$ is a $\lambda$-eigenfunction of $G$ and $\varphi\in \mathrm{Aut}(G)$, we have $h\in U_{\lambda}(G)$.
The restriction of $h$ to $V_1$ is denoted by $h_1$.

Let us prove that $h_1\in U_{\lambda+1}(G_1)$. Let us consider a vertex $x\in V_1$.
Since $h\in U_{\lambda}(G)$, we have $$\lambda\cdot h(x)=\sum_{y\in{N_G(x)}}h(y).$$
Then 
\[
\begin{gathered}
\lambda\cdot h(x)=\sum_{y\in{N_G(x)\cap V_1}}h(y)+\sum_{y\in{N_G(x)\cap V_2}}h(y)+\sum_{y\in{N_G(x)\cap V_3}}h(y)= \\
=\sum_{y\in{N_{G_1}(x)}}h(y)+h(\varphi(x))
\end{gathered}
\]
Hence we obtain that $$(\lambda+1)\cdot h(x)=\sum_{y\in{N_{G_1}(x)}}h(y).$$
Therefore, $h_1\in U_{\lambda+1}(G_1)$.
Finally, note that $f_{P,\varphi_1,\varphi_2}=h_1(\varphi_1^{-1})$. Since $h_1\in U_{\lambda+1}(G_1)$ and $\varphi_1$ is an isomorphism of $G_1$ and $G_0$, we obtain that
$f_{P,\varphi_1,\varphi_2}\in U_{\lambda+1}(G_0)$.
\end{proof}

\section{Examples}\label{Sec:Examples}
In this section, we discuss how to apply Theorem \ref{Th:Reduction} to the Hamming graph, the Johnson graph and the halved $n$-cube. In particular, we show that these graphs admit special pairs.
\subsection{Hamming graph}\label{SubSec:Hamming }
The {\em  Hamming graph} $H(n,q)$ is defined as follows. The vertex set of $H(n,q)$ is $\mathbb{Z}_{q}^n$, 
and two vertices are adjacent if they differ in exactly one coordinate.

Let $k,m\in \mathbb{Z}_q$, $k\neq m$ and $r\in \{1,2,\ldots,n\}$.
Denote $$V_1=\{x\in\mathbb{Z}_{q}^n:x_r=k\},$$ 
$$V_2=\{x\in\mathbb{Z}_{q}^n:x_r=m\}$$ and
$V_3=\mathbb{Z}_{q}^n\setminus (V_1\cup V_2)$. Denote $X=\{V_1,V_2,V_3\}$. 

Define a map $\varphi:\mathbb{Z}_{q}^n\longrightarrow \mathbb{Z}_{q}^n$ as follows:
$$\varphi(x_1,\ldots,x_n)=(x_1,\ldots,x_{r-1},(km)(x_r),x_{r+1},\ldots,x_n).$$
Note that $(\varphi,X)$ is a special pair of $H(n,q)$.

Let $G_0=H(n-1,q)$. Define maps $\varphi_1:V_1\longrightarrow V(G_0)$ and $\varphi_2:V_2\longrightarrow V(G_0)$ as follows:
$$\varphi_1(x)=\Delta_{r}(x)$$ and $$\varphi_2(y)=\Delta_{r}(y).$$
One can check that $G[V_1]$ and $G[V_2]$ are isomorphic to $G_0$ and $\varphi_1$ and $\varphi_2$ are the corresponding isomorphisms.
Thus, $(\varphi,X)$, $G_0$, $\varphi_1$ and $\varphi_2$ satisfy the conditions of Theorem \ref{Th:Reduction}.

\subsection{Johnson graph}\label{SubSec:Johnson }
The {\em Johnson graph} $J(n,k)$ is defined as follows. The vertex set of $J(n,k)$ is $\{x\in\mathbb{Z}_{2}^n:|x|=k\}$, 
and two vertices are adjacent if they differ in exactly two coordinates.

Let $i,j\in \{1,2,\ldots,n\}$ and $i<j$.
Denote $$V_1=\{x\in\mathbb{Z}_{2}^n:|x|=k,x_i=1,x_j=0\},$$ 
$$V_2=\{x\in\mathbb{Z}_{2}^n:|x|=k,x_i=0,x_j=1\}$$ and
$V_3=V(J(n,k))\setminus (V_1\cup V_2)$. Denote $X=\{V_1,V_2,V_3\}$.

Define a map $\varphi:V(J(n,k))\longrightarrow V(J(n,k))$ as follows:
$$\varphi(x)=\pi_{i,j}(x).$$
Note that $(\varphi,X)$ is a special pair of $J(n,k)$.

Let $G_0=J(n-2,k-1)$. Define maps $\varphi_1:V_1\longrightarrow V(G_0)$ and $\varphi_2:V_2\longrightarrow V(G_0)$ as follows:
$$\varphi_1(x)=\Delta_{i,j}(x)$$ and $$\varphi_2(y)=\Delta_{i,j}(y).$$
One can check that $G[V_1]$ and $G[V_2]$ are isomorphic to $G_0$ and $\varphi_1$ and $\varphi_2$ are the corresponding isomorphisms.
Thus, $(\varphi,X)$, $G_0$, $\varphi_1$ and $\varphi_2$ satisfy the conditions of Theorem \ref{Th:Reduction}.

\subsection{Halved $n$-cube}\label{SubSec:Halved cube}
The {\em halved $n$-cube} $\frac{1}{2}H(n)$ is defined as follows. The vertex set of $\frac{1}{2}H(n)$ is $\{x\in\mathbb{Z}_{2}^n:|x|~\text{is even}\}$, 
and two vertices are adjacent if they differ in exactly two coordinates.

Let $i,j\in \{1,2,\ldots,n\}$ and $i<j$.
Denote $$V_1=\{x\in\mathbb{Z}_{2}^n:|x|~\text{is even},x_i=1,x_j=0\},$$ 
$$V_2=\{x\in\mathbb{Z}_{2}^n:|x|~\text{is even},x_i=0,x_j=1\}$$ and
$V_3=V(\frac{1}{2}H(n))\setminus (V_1\cup V_2)$. Denote $X=\{V_1,V_2,V_3\}$.

Define a map $\varphi:V(\frac{1}{2}H(n))\longrightarrow V(\frac{1}{2}H(n))$ as follows:
$$\varphi(x)=\pi_{i,j}(x).$$
Note that $(\varphi,X)$ is a special pair of $\frac{1}{2}H(n)$.

We define a graph $G_0$  as follows. The vertex set of $G_0$ is $\{x\in\mathbb{Z}_{2}^{n-2}:|x|~\text{is odd}\}$, 
and two vertices are adjacent if they differ in exactly two coordinates. 
Note that $G_0$ is isomorphic to $\frac{1}{2}H(n-2)$.
Define maps $\varphi_1:V_1\longrightarrow V(G_0)$ and $\varphi_2:V_2\longrightarrow V(G_0)$ as follows:
$$\varphi_1(x)=\Delta_{i,j}(x)$$ and $$\varphi_2(y)=\Delta_{i,j}(y).$$
One can check that $G[V_1]$ and $G[V_2]$ are isomorphic to $G_0$ and $\varphi_1$ and $\varphi_2$ are the corresponding isomorphisms.
Thus, $(\varphi,X)$, $G_0$, $\varphi_1$ and $\varphi_2$ satisfy the conditions of Theorem \ref{Th:Reduction}.

\section{Acknowledgements}
The author is grateful to Sergey Goryainov and Ivan Mogilnykh for helpful discussions.

\end{document}